\documentclass[letterpaper, 10 pt, conference]{ieeeconf}  

\IEEEoverridecommandlockouts                              
\usepackage{amsmath,amsfonts,amssymb,color}                                                                              
\usepackage{epsfig}
\usepackage{psfrag}
\usepackage{algorithm}
\usepackage{algpseudocode}
\usepackage{array}
\usepackage{epstopdf}
\usepackage{cite}
\usepackage{footmisc}
\usepackage{mathtools}
\usepackage{bm}
\usepackage[utf8]{inputenc}

\newtheorem{problem}{Problem}

\newtheorem{theorem}{Theorem}
\newtheorem{corollary}{Corollary}
\newtheorem{lemma}{Lemma}
\newtheorem{remark}{Remark}
\newtheorem{definition}{Definition}
\newtheorem{proposition}{Proposition}
\newtheorem{example}{Example}
\newtheorem{assumption}{Assumption}
\renewcommand{\theenumi}{(\alph{enumi}}

\title{\LARGE \bf
Distributed Maximization of Submodular and Approximately Submodular Functions
}

\author{Lintao Ye and Shreyas Sundaram
\thanks{This research was supported by NSF grant CMMI-1635014. Lintao Ye and Shreyas Sundaram are with the School of Electrical
and Computer Engineering at Purdue University. Email: \{ye159,sundara2\}@purdue.edu. 
}
}

\begin{document}

\maketitle
\thispagestyle{plain}
\pagestyle{plain}

\begin{abstract}
We study the problem of maximizing a submodular function, subject to a cardinality constraint, with a set of agents communicating over a connected graph. We propose a distributed greedy algorithm that allows all the agents to converge to a near-optimal solution to the global maximization problem using only local information and communication with neighbors in the graph. The near-optimal solution approaches the $(1-1/e)$ approximation of the optimal solution to the global maximization problem with an additive factor that depends on the number of communication steps in the algorithm. We then analyze convergence guarantees of the proposed algorithm. This analysis reveals a tradeoff between the number of communication steps and the performance of the algorithm. Finally, we extend our analysis to nonsubmodular settings, using the notion of approximate submodularity.
\end{abstract}

\section{Introduction}
In recent years, the analysis of large-scale networks has received much attention from researchers, where the networks consist of a group of agents with different local objective functions. For such networks, the goal is to design a resource allocation method that operates in a decentralized way with local communication and fast convergence to an (approximately) optimal operating point. Scenarios where the local objective functions depend on the entire resource allocation vector are of particular interest. For such scenarios, there is a vast literature on designing  distributed algorithms that guarantee convergence of the solution obtained by each agent to an optimizer of the average of all the local objective functions (e.g., \cite{tsitsiklis1986distributed,xiao2006optimal,nedic2009distributed,nedic2010constrained,yuan2016convergence} and the references therein). 

Much of the existing work has been devoted to optimization problems in the continuous domain, where  the local objective functions are convex. In contrast, settings with (discrete) submodular objective functions have been less explored (e.g.,  \cite{mokhtari2018decentralized,gharesifard2017distributed,mirzasoleiman2013distributed}). Nonetheless, the problem of maximizing submodular functions (subject to constraints) arises in many different applications including, for instance, budget allocation \cite{soma2014optimal}, sensor placement \cite{krause2008near}, sensor scheduling \cite{jawaid2015submodularity} and influence maximization in social networks \cite{kempe2003maximizing}. Thus, in this paper we focus on scenarios in distributed optimization where the local objective functions are submodular. 

\subsection*{Related Work}
In \cite{mokhtari2018decentralized}, the authors considered maximizing a discrete submodular function subject to a general matroid constraint and proposed a decentralized algorithm to solve this problem. The algorithm relies on first lifting the local discrete submodular functions to continuous domains and then applying appropriate rounding schemes to the obtained solution. The authors in \cite{gharesifard2017distributed} considered the scenario where a group of agents sequentially maximize a submodular function. The problem reduces to the canonical problem of maximizing a submodular function subject to a partitioned matroid constraint, which can be solved by greedy algorithms with a $1/2$ multiplicative approximation ratio \cite{fisher1978analysis}.  In contrast, we consider the setting where a group of agents maximize a (global) submodular function (subject to a cardinality constraint) in parallel using a decentralized (i.e., distributed) greedy algorithm, which does not require any lifting or rounding process.

Regarding the maximization of nonsubmodular functions, the notion of approximate submodularity has been used to provide performance guarantees for (centralized) greedy algorithms applied to such problems (e.g., \cite{das2018approximate ,bian2017guarantees,ye2019sensor}). Here, we aim to propose a distributed greedy algorithm that can also solve such problems in a distributed manner.

\subsection*{Contributions}
We propose a distributed greedy algorithm with a group of agents communicating over a connected network, which allows each agent to converge to within an additive factor of the $(1-1/e)$ approximation of the optimal solution to the global optimization problem.  This additive factor is a function of the number of the agents, the cardinality constraint, properties of the local functions and design parameters in the distributed greedy algorithm. In particular, the analysis reveals a tradeoff between the performance of the algorithm and the number of communication steps in the algorithm. Finally, we extend our analysis to cases when the objective function is approximately submodular.


\subsection*{Notation and terminology}
The sets of integers and real numbers are denoted as $\mathbb{Z}$ and $\mathbb{R}$, respectively. For $x\in\mathbb{R}$, let $|x|$ denote its absolute value. For a set $\mathcal{S}$, let $|\mathcal{S}|$ denote its cardinality. Let $\mathbf{1}_n$ denote a column vector of dimension $n$ with all of its elements equal to 1.  For a matrix $P\in\mathbb{R}^{m\times n}$  and a vector $y\in\mathbb{R}^n$, let $P^{\prime}$ and $y^{\prime}$ be their transposes, respectively. Let $P_{ij}$ denote the element in the  $i$th row and $j$th column of $P$. Let $P_i$ denote the $i$th row of $P$. The eigenvalues of $P$ are ordered with nonincreasing magnitude (i.e., $|\lambda_1(P)|\ge\cdots\ge|\lambda_n(P)|$). Given two functions $\varphi_1:\mathbb{R}_{\ge0}\to\mathbb{R}$ and $\varphi_2:\mathbb{R}_{\ge0}\to\mathbb{R}$, $\varphi_1(n)$ is $O(\varphi_2(n))$ if there exist positive constants $c$ and $N$ such that $|\varphi_1(n)|\le c|\varphi_2(n)|$ for all $n\ge N$.

\section{Problem Formulation} \label{sec:problem formulation}
We first introduce the following definitions (e.g., \cite{nemhauser1978analysis}).

\begin{definition}
\label{def:increasing}
Given a ground set $V$, a set function $f:2^V\to\mathbb{R}_{\ge0}$ is said to be monotone nondecreasing if for all $A\subseteq B\subseteq V$, $f(A)\le f(B)$.
\end{definition}

\begin{definition}
\label{def:submodular}
Given a set $V$, a set function $f:2^V\to\mathbb{R}_{\ge0}$ is said to be submodular if for all $A\subseteq B\subseteq V$ and for all $v\in V\setminus B$, $f(\{v\}\cup A)-f(A)\ge f(\{v\}\cup B)-f(B)$.
\end{definition}

Consider a set of $n$ agents that communicate over a graph to maximize a global objective function. Each agent can be viewed as a node $i\in\mathcal{N}\triangleq\{1,\dots,n\}$. Denote the communication graph as $\mathcal{G}=(\mathcal{N}, \mathcal{E})$, which is assumed to be undirected and connected throughout this paper.  An edge $(i,j)\in\mathcal{E}$, which is an unordered pair, indicates a {\it bidirectional} communication between agents $i$ and $j$, i.e., agent $i$ can receive information from agent $j$ at each time step, and vice versa. Let $\mathcal{N}_i$ denote the set of neighbors of agent $i$, i.e., $\mathcal{N}_i\triangleq\{j\in\mathcal{N}:(i,j)\in\mathcal{E}\}$. 

\begin{definition}
\label{def:diameter of graph}
The diameter of a connected graph $\mathcal{G}=(\mathcal{N},\mathcal{E})$ is given by $d(\mathcal{G})=\mathop{\max}_{i,j\in\mathcal{N}}l_{ij}$, where $l_{ij}$ is the length of the shortest path (i.e., distance) from $i$ to $j$ in $\mathcal{G}$.
\end{definition}

We now consider the scenario where each agent $i\in\mathcal{N}$ has access to a local set function $f_i:2^V\to\mathbb{R}_{\ge0}$ with a cardinality constraint $K\in\mathbb{Z}_{\ge1}$. The function $f_i(\cdot)$ is assumed to be monotone nondecreasing and submodular for all $i\in\mathcal{N}$.\footnote{We will generalize our analysis to monotone nondecreasing nonsubmodular functions later.} Meanwhile, we assume without loss of generality  that $f_i(\cdot)$ is normalized such that $f_i(\emptyset)=0$ for all $i\in\mathcal{N}$. The objective for the agents is to solve, in a distributed manner (i.e., by repeatedly exchanging information only with their neighbors), the following global optimization problem:
\begin{equation}
\label{eqn:global objective function}
\mathop{\max}_{S\subseteq V, |S|\le K} f(S)=\mathop{\max}_{S\subseteq V, |S|\le K} \frac{1}{n}\sum_{i=1}^n f_i(S),
\end{equation}
where $f(S)\triangleq\frac{1}{n}\sum_{i=1}^n f_i(S)$, and $K\in\mathbb{Z}_{\ge1}$. 

Since the sum of monotone nondecreasing submodular functions is monotone nondecreasing submodular, $f(\cdot)$ is monotone nondecreasing submodular (with $f(\emptyset)=0$). Thus, the global optimization problem (Problem \eqref{eqn:global objective function}) is to maximize a monotone nondecreasing submodular function subject to a cardinality constraint. A (centralized) greedy algorithm has been proposed to solve Problem \eqref{eqn:global objective function} with a multiplicative approximation ratio of $(1-1/e)$ \cite{nemhauser1978analysis}. Moreover, it was shown in \cite{feige1998threshold} that the greedy algorithm achieves the best possible approximation ratio of any polynomial-time approximation algorithm for Problem  \eqref{eqn:global objective function} in the centralized case  if P$\neq$NP. This motivates us to extend the centralized greedy algorithm to solve Problem  \eqref{eqn:global objective function}  in the distributed case. 

\section{Distributed Greedy Algorithm}
The main idea of the distributed greedy algorithm (Algorithm~$\ref{algorithm:decentralized greedy}$) is as follows. Based on the centralized greedy algorithm, the distributed greedy algorithm runs for $K$ rounds in total. Given a current set $S_k$ ($|S_k|=k$) of selected elements before the $(k+1)$th ($k\le K-1$) round  of the algorithm, the $(k+1)$th round of the algorithm lets all the agents reach consensus at an element $s^{\prime}\in V\setminus S_{k}$ (after a certain number of communication steps) that maximizes $(f(\{s^{\prime}\}\cup S_k)-f(S_k))$ with some additive error (suboptimality) and add $s^{\prime}$ to $S_{k}$ to obtain $S_{k+1}$.  After $K$ rounds, all the agents obtain a set $S_{K}$ ($|S_{K}|=K$) that gives a solution to Problem \eqref{eqn:global objective function} with a suboptimality bound which we will discuss later. 

\begin{algorithm}
\textbf{Input:} $f_i:2^V\to\mathbb{R}_{\ge0}$, $K\in\mathbb{Z}_{\ge1}$, $\mathcal{G}=(\mathcal{N},\mathcal{E})$\\
\textbf{Design parameters:} $W\in\mathbb{R}^{n\times n}$, $T\in\mathbb{Z}_{\ge1}$, $T'=T+1+d(\mathcal{G})$, $\psi\in\mathbb{R}_{\ge0}$\\
\textbf{Output:} $\bar{S}^K$
\caption{Distributed Greedy Algorithm}\label{algorithm:decentralized greedy}
\begin{algorithmic}[1]
\State $k=0$, $\bar{S}_i^0 = \emptyset$
\While{$k\le K-1$}
    \State Calculate $x_{i,v}^0=f_i(\{v\}\cup \bar{S}_i^k)-f_i(\bar{S}_i^k),\ \forall v\in V\setminus\bar{S}_i^k$
    \State Stack $x_{i,v}^0$ into a vector $x_i^0\in\mathbb{R}^{|V\setminus\bar{S}_i^k|}$
    \For{$t =0,\dots, T-1$}
    \State Update $x_i^{t+1}=w_{ii}x_i^t+\sum_{j\in\mathcal{N}_i}w_{ij}x_j^t$
    \EndFor
    \State Find $x_{i,v_{i^*}}^{T}=\mathop{\max}_{v\in V}x_{i,v}^{T}$ 
    \State Set $S_i^{T+1}=\{v: x_{i,v}^{T}\ge x_{i,v_{i^*}}^{T}-\psi\}$
    \For{$t=T+1,\dots,T'-1$}
    \State Update $S_i^{t+1}=\bigcap_{j\in\mathcal{N}_i}S_j^t$
    \EndFor
    \State Find $j_{\text{min}}=\mathop{\min}\{j:v_j\in S_i^{T'}\}$
    \State Update $\bar{S}_i^{k+1}=\{v_{j_{\text{min}}}\}\cup\bar{S}_i^k$
    \State $k=k+1$
\EndWhile
\State $\bar{S}^K=\bar{S}_i^K$
\end{algorithmic}
\end{algorithm}

Specifically, denote $V\triangleq\{v_1,\dots,v_{|V|}\}$. Each agent $i\in\mathcal{N}$ initializes a local variable $\bar{S}_i^k=\emptyset$ with $k=0$. In the first round of the distributed greedy algorithm (with $k=0$), each agent $i\in\mathcal{N}$ starts with $t=0$, and calculates $x_{i,v}^0= f_i(\{v\})-f_i(\emptyset)$ for all $v\in V$, where $f_i(\emptyset)=0$. In other words, each agent $i\in\mathcal{N}$ maintains a local variable $x_i^t=\begin{bmatrix}x^t_{i,v_1} & \cdots & x^t_{i,v_{|V|}}\end{bmatrix}^{\prime}\in\mathbb{R}^{|V|}$ at time step $t=0$. Given design parameter $T\in\mathbb{Z}_{\ge1}$ of the algorithm, agent $i$ first updates $x_i^t$ (using $x_i^t$ and $x_j^t$ for all $j\in\mathcal{N}_i$) from time step $t=0$ to time step $t=T$. Then, agent $i$ obtains a set $S_i^{T+1}$ and maintains a local variable $S_i^t$. Given design parameter $T'\in\mathbb{Z}_{\ge1}$ of the algorithm, where $T'>T+1$,\footnote{We will explain the choice of $T'$ in the algorithm later.} agent $i$ now updates $S_i^t$ (using $S_j^t$ for all $j\in\mathcal{N}_i$) from time step $t=T+1$ to time step $t=T'$. At the end of the first round of the distributed greedy algorithm (i.e., at $t=T'$), all the agents in $\mathcal{N}$ choose the same element $s_0$ from $S_i^{T'}$ and update $\bar{S}_i^0$ as $\bar{S}_i^1=\bar{S}_i^0\cup\{s_0\}$. After finishing the first round, the distributed greedy algorithm enters the second round (with $k=1$). Again, each agent $i\in\mathcal{N}$ starts with $t=0$ and calculates $x_i^0=\begin{bmatrix}x^0_{i,v_{j_1}} & \cdots & x^0_{i,v_{j_{|V_1|}}}\end{bmatrix}^{\prime}\in\mathbb{R}^{|V_1|}$, where $V_1\triangleq V\setminus\bar{S}_i^1=\{v_{j_1},\dots,v_{j_{|V_1|}}\}$ and $x_{i,v}^0= f_i(\{v\}\cup\bar{S}_i^1)-f_i(\bar{S}_i^1)$ for all $v\in V_1$.  Similarly to the first round of the distributed greedy algorithm, agent $i$ obtains the updated local variables $x_i^{T}$ and $S_i^{T'}$, which leads to the update of $\bar{S}_i^1$ as $\bar{S}_i^2=\bar{S}_i^1\cup \{s_1\}$, where $s_1$ is an element chosen from $S_i^{T'}$ by all the agents in $\mathcal{N}$. The distributed greedy algorithm repeats the above process for $K$ rounds. Note from the above arguments that for all $k\in\{0,\dots,K-1\}$, $\bar{S}_i^k=\bar{S}_j^k$ for all $i,j\in\mathcal{N}$.
\begin{remark}
As we will see in the following, our distributed greedy algorithm has a consensus phase (with a limited number of communications) among the set of agents in each round of the algorithm. Similar consensus-based distributed algorithms have been used in, for example, distributed task allocation \cite{choi2009consensus} and distributed Kalman filtering \cite{carli2008distributed}. 
\end{remark}

To implement Algorithm~$\ref{algorithm:decentralized greedy}$, we assume the following.
\begin{assumption}
\label{assumption:central clock}
The set $\mathcal{N}$ of agents has a synchronized clock such that all the agents in $\mathcal{N}$ know the current values of $k\in\mathbb{Z}_{\ge0}$ and $t\in\mathbb{Z}_{\ge1}$. Moreover, the agents in $\mathcal{N}$ know the design parameters $T$ and $T'$  before running the algorithm, where $T,T',K\in\mathbb{Z}_{\ge1}$ and $T'>T+1$.
\end{assumption}

We now describe (distributed) update rules for the agents in $\mathcal{N}$ to update the local variables $x_i^t$, $S_i^t$ and $\bar{S}_i^k$ from Algorithm~$\ref{algorithm:decentralized greedy}$, respectively. Consider the $(k+1)$th round of the distributed algorithm, where $k\le K-1$ ($k\in\mathbb{Z}_{\ge0}$). Noting that $\bar{S}_i^k=\bar{S}_j^k$ for all $i,j\in\mathcal{N}$, we denote $\bar{S}^k=\bar{S}_i^k$. 

First, we describe the update rule for $x_i^t\in\mathbb{R}^{|V_k|}$, where $V_k\triangleq V\setminus\bar{S}_i^k$. Note that each agent $i\in\mathcal{N}$ initializes $x_i^0$ in lines $3$-$4$ of Algorithm~$\ref{algorithm:decentralized greedy}$ as
\begin{equation}
\label{eqn:initial x_t}
x_i^0=\begin{bmatrix}x^0_{i,v_{l_1}} & \cdots & x^0_{i,v_{l_{|V_k|}}}\end{bmatrix}^{\prime}\in\mathbb{R}^{|V_k|},
\end{equation}
where $x_{i,v}^0= f_i(\{v\}\cup \bar{S}_i^k)-f_i(\bar{S}_i^k)$ for all $v\in V_k$ and $V_k=\{v_{l_1},\dots,v_{l_{|V_k|}}\}$. Since $f_i(\cdot)$ is monotone nondecreasing for all $i\in\mathcal{N}$, we have $x_{i,v}^0\ge 0$ for all $i\in\mathcal{N}$ and for all $v\in V_k$. Starting from time step $t=0$ with $x_i^0$, each agent $i\in\mathcal{N}$ computes $x_i^{t+1}$ in line 6 of Algorithm~$\ref{algorithm:decentralized greedy}$ according to the following update rule:
\begin{equation}
\label{eqn:decentralized updating rule}
x_i^{t+1}=w_{ii}x_i^t+\sum_{j\in\mathcal{N}_i}w_{ij}x_j^{t},
\end{equation}
where agent $i$ assigns a weight $w_{ij}\in\mathbb{R}$ to agent $j$ for all $j\in\mathcal{N}_i\cup\{i\}$ and updates $x_i^{t+1}$ as a weighted average of $x_j^t$ from $j\in\mathcal{N}_i\cup\{i\}$. Denote $W\in\mathbb{R}^{n\times n}$ as the weight matrix (or {\it mixing matrix}) such that $W_{ij}=w_{ij}$ for all $i,j\in\mathcal{N}$. We assume that the weight matrix $W$ satisfies the following assumptions, which are standard in the distributed  optimization literature (e.g., \cite{yuan2016convergence}).
\begin{assumption}
\label{assumption:assumption of W}
The weight matrix $W\in\mathbb{R}^{n\times n}$ (associated with $\mathcal{G}=(\mathcal{N},\mathcal{E})$) is assumed to satisfy: (1) $W_{ij}\in\mathbb{R}_{\ge0}$ for all $i,j\in\mathcal{N}$ and $W_{ij}=0$ if $(i,j)\notin\mathcal{E}$; (2) $W\mathbf{1}_n=\mathbf{1}_n$; (3) $W=W'$  and (4) $\mu(W)\triangleq\mathop{\max}\{\lambda_2({W}),-\lambda_n(W)\}<1$. 
\end{assumption}
\begin{remark}

Assumption~$\ref{assumption:assumption of W}.(1)$-$(3)$ ensure that $W$ is symmetric and doubly stochastic. Thus, the eigenvalues of $W$ are real and satisfy $1=\lambda_1(W)\ge\lambda_2(W)\ge\cdots\ge\lambda_n(W)\ge-1$ (e.g., \cite{boyd2004fastest}). Assumption~$\ref{assumption:assumption of W}.(4)$ is satisfied if the Markov chain corresponding to matrix $W$ is irreducible and aperiodic (e.g., \cite{boyd2004fastest}, \cite{diaconis1991geometric}). Note that the weight matrix $W$ is also a design parameter of the distributed greedy algorithm. Similarly to Assumption~$\ref{assumption:central clock}$, we assume that each agent $i\in\mathcal{N}$ knows $W_i$ (i.e., $w_{ij}$ for all $j\in\mathcal{N}\cup\{i\}$).
\end{remark}

By repeatedly running update rule \eqref{eqn:decentralized updating rule},  $x_i^t$ will converge to $\frac{1}{n}\sum_{j\in\mathcal{N}}x_j^0$ as $t\to\infty$ (e.g., \cite{boyd2004fastest}) for all $i\in\mathcal{N}$, so that $x_{i,v}^t$ will converge to $\frac{1}{n}\sum_{j\in\mathcal{N}}x_{j,v}^0$ as $t\to\infty$ for all $v\in V_k$. To analyze finite-time performance of update rule \eqref{eqn:decentralized updating rule},  let us first consider the following function of $t\in\mathbb{Z}_{\ge1}$: 
\begin{equation}
\label{eqn:definition of consensus error}
\delta_k(t)=\mathop{\max}_{i\in\mathcal{N},v\in V_k}\Big|x_{i,v}^t-\frac{1}{n}\sum_{j\in\mathcal{N}}x^0_{j,v}\Big|,
\end{equation}
which is the maximum deviation of $x_{i,v}^t$ from the limiting value $\frac{1}{n}\sum_{j\in\mathcal{N}}x^0_{j,v}$ over all agents $i\in\mathcal{N}$ and all elements $v\in V_k$ at any (finite) time step $t$. Moreover, note that for each agent $i\in\mathcal{N}$ and an element $v\in V_k$, we can view $x^t_{i,v}$ as an estimate of $\frac{1}{n}\sum_{j\in\mathcal{N}}(f_j(\{v\}\cup\bar{S}^k)-f_j(\bar{S}^k))$ at time step $t$. Thus, $\delta_k(t)$ captures the maximum error (in absolute value) of such estimates over all $i\in\mathcal{N}$ and all $v\in V_k$ at time step $t$. We will use the following result (e.g., \cite{diaconis1991geometric}).

\begin{lemma}
\label{lemma:finite time convergence of W}
Consider a weight matrix $W\in\mathbb{R}^{n\times n}$ that satisfies Assumption~$\ref{assumption:assumption of W}$. The following inequality holds:
\begin{equation*}
\mathop{\max}_{i\in\mathcal{N}}\sum_{j\in\mathcal{N}}\Big|(W^t)_{ij}-\frac{1}{n}\Big|\le\sqrt{n}(\mu(W))^t,
\end{equation*}
where $\mu(W)=\mathop{\max}\{\lambda_2({W}),-\lambda_n(W)\}$.
\end{lemma}

We then have the following results.
\begin{lemma}
\label{lemma:upper bound on the estimate error}
Consider the update rule \eqref{eqn:decentralized updating rule} initialized with $x_i^0$ given by Eq.~\eqref{eqn:initial x_t}. Suppose Assumptions~$\ref{assumption:central clock}$ and $\ref{assumption:assumption of W}$ hold. For all time steps $t\in\mathbb{Z}_{\ge1}$, the following inequality holds:
\begin{equation}
\label{eqn:bound on delta kt}
\delta_k(t) \le \sqrt{n}(\mu(W))^t F_h,
\end{equation}
where $\delta_k(t)$ is defined in Eq.~\eqref{eqn:definition of consensus error} and $F_h\triangleq\mathop{\max}_{i\in\mathcal{N}}f_i(V)$.
\end{lemma}
\begin{proof}
Denoting $\bar{x}^t_v=\begin{bmatrix}x^t_{1,v} & \cdots & x^t_{n,v}\end{bmatrix}^{\prime}$ for all $v\in V_k$ and for all $t\in\mathbb{Z}_{\ge0}$, we have from Eq.~\eqref{eqn:decentralized updating rule} $\bar{x}^{t+1}_v=W\bar{x}^t_v$, which implies $\bar{x}^t_v=W^t\bar{x}^0_v$. It then follows that 
\begin{align}\nonumber
\delta_k(t)&=\mathop{\max}_{i\in\mathcal{N},v\in V_k}\Big|(W^t)_i\bar{x}^0_v-\frac{1}{n}\sum_{j\in\mathcal{N}}x^0_{j,v}\Big|\\\nonumber
&=\mathop{\max}_{i\in\mathcal{N},v\in V_k}\Big|\sum_{j\in\mathcal{N}}(W^t)_{ij}x_{j,v}^0-\frac{1}{n}\sum_{j\in\mathcal{N}}x^0_{j,v}\Big|\\\nonumber
&=\mathop{\max}_{v\in V_k}\mathop{\max}_{i\in\mathcal{N}}\Big|\sum_{j\in\mathcal{N}}((W^t)_{ij}-\frac{1}{n})x^0_{j,v}\Big|\\\nonumber
&\le\mathop{\max}_{v\in V_k}\mathop{\max}_{i\in\mathcal{N}}\big(\sum_{j\in\mathcal{N}}\Big|(W^t)_{ij}-\frac{1}{n}\Big|x^0_{j,v}\big)\\
&\le\sqrt{n}(\mu(W))^t\mathop{\max}_{j\in \mathcal{N}}\mathop{\max}_{v\in V_k}x_{j,v}^0\le \sqrt{n}(\mu(W))^tF_h,\label{eqn:bound on delta_1^t}
\end{align}
where the first inequality in \eqref{eqn:bound on delta_1^t} follows from Lemma~$\ref{lemma:finite time convergence of W}$. For the second inequality in \eqref{eqn:bound on delta_1^t}, we note that $x_{j,v}^0=f_j(\{v\}\cup\bar{S}_i^k)-f_j(\bar{S}_i^k)\le f_j(V)$ for all $j\in\mathcal{N}$ and for all $v\in V_k$, since $f_j(\cdot)$ is monotone nondecreasing with $f_j(\emptyset)=0$.\footnote{Noting that $f_j(\{v\}\cup\bar{S}_i^k)-f_j(\bar{S}_i^k)\le f_j(v)-f_j(\emptyset)\le f_j(V)$ $\forall j\in\mathcal{N}$ and $\forall v\in V_k$ by the submodularity of $f_j(\cdot)$, the bound in \eqref{eqn:bound on delta kt} can potentially be tightened by defining $F'_h\triangleq\mathop{\max}_{i\in\mathcal{N}}\mathop{\max}_{v\in V}f_i(v)$.} Thus, we have $\mathop{\max}_{j\in \mathcal{N}}\mathop{\max}_{v\in V_k}x_{j,v}^0\le\mathop{\max}_{j\in\mathcal{N}}f_j(V)= F_h$.
\end{proof}

\begin{lemma}
\label{lemma:elements close to the maximum}
Consider the update rule \eqref{eqn:decentralized updating rule} initialized with $x_i^0$ given by Eq.~\eqref{eqn:initial x_t}. Suppose Assumptions~$\ref{assumption:central clock}$ and $\ref{assumption:assumption of W}$ hold. For each time step $t\in\mathbb{Z}_{\ge1}$, denote  $x_{i,v_{i^*}}^t=\mathop{\max}_{v\in V_k}x_{i,v}^t$, where $v_i^*\triangleq\mathop{\arg}{\max}_{v\in V_k}x_{i,v}^t$, for all $i\in\mathcal{N}$. The following holds:
\begin{equation}
x_{i,v_{j^*}}^t\ge x_{i,v_{i^*}}^t-4\epsilon(t),\forall i,j\in\mathcal{N},
\end{equation}
where 
\begin{equation}
\label{eqn:def of epsilon 1}
\epsilon(t)\triangleq\sqrt{n}(\mu(W))^tF_h
\end{equation}
 is a function of $t\in\mathbb{Z}_{\ge1}$, and $F_h=\mathop{\max}_{i\in\mathcal{N}}f_i(V)$.
\end{lemma}
\begin{proof}
Consider any time step $t\in\mathbb{Z}_{\ge1}$, and any two (distinct) agents $i,j\in\mathcal{N}$. We have from Lemma~$\ref{lemma:upper bound on the estimate error}$
\begin{equation}
\label{eqn:inequality for derivation 1}
\Big|x_{i,v_{i^*}}^t-\frac{1}{n}\sum_{q\in\mathcal{N}}x^0_{q,v_{i^*}}\Big|\le\epsilon(t),
\end{equation}
and
\begin{equation} 
\label{eqn:inequality for derivation 2}
\Big|x_{j,v_{i^*}}^t-\frac{1}{n}\sum_{q\in\mathcal{N}}x^0_{q,v_{i^*}}\Big|\le\epsilon(t).
\end{equation}
It then follows from \eqref{eqn:inequality for derivation 1}-\eqref{eqn:inequality for derivation 2} that
\begin{equation} 
\label{eqn:inequality for derivation 3}
\Big|x_{i,v_{i^*}}^t-x_{j,v_{i^*}}^t\Big|\le2\epsilon(t).
\end{equation}
Similarly, we have
\begin{equation} 
\label{eqn:inequality for derivation 4}
\Big|x_{j,v_{j^*}}^t-x_{i,v_{j^*}}^t\Big|\le2\epsilon(t).
\end{equation}
Therefore, we have the following:
\begin{equation*}
\label{eqn:eqn:inequality for derivation 5}
x_{i,v_{i^*}}^t-x_{i,v_{j^*}}^t\le x_{j,v_{i^*}}^t+2\epsilon(t)-x_{j,v_{j^*}}^t+2\epsilon(t)\le4\epsilon(t),
\end{equation*}
where the first inequality follows from  \eqref{eqn:inequality for derivation 3}-\eqref{eqn:inequality for derivation 4} and the second inequality follows from the fact $x_{j,v_{j^*}}^t=\mathop{\max}_{v\in V_k}x_{j,v}^t$, i.e. $x_{j,v_{i^*}}^t\le x_{j,v_{j^*}}^t$. 
\end{proof}

Note that each agent $i\in\mathcal{N}$ updates $x_i^t$ from time step $t=0$ to time step $t=T$, where we recall from Assumption~$\ref{assumption:central clock}$ that $T\in\mathbb{Z}_{\ge1}$ is a design parameter of the algorithm that is known to all the agents in $\mathcal{N}$. 

Next, we describe the update rule for $S_i^t$. Specifically, after running update rule \eqref{eqn:decentralized  updating rule} from time step $t=0$ to time step $t=T$ and obtaining $x_i^T$, each agent $i\in\mathcal{N}$ obtains $S_i^{T+1}$ as
\begin{equation}
\label{eqn:set of candidate updating elements}
S_i^{T+1}=\{v: x_{i,v}^T\ge x_{i,v_{i^*}}^T-\psi\}, 
\end{equation}
where $\psi\in\mathbb{R}_{\ge0}$ is a design parameter of the algorithm that needs to satisfy the following condition:
\begin{equation}
\label{eqn:cond on epsilon}
\psi\ge4\sqrt{n}(\mu(W))^TF_h.
\end{equation}
Note that we also assume that each agent $i\in\mathcal{N}$ knows the design parameter $\psi$. We then see from Lemma~$\ref{lemma:elements close to the maximum}$ and update rule \eqref{eqn:set of candidate updating elements} with condition \eqref{eqn:cond on epsilon} that $v_{j^*}\in S_i^{T+1}$ for all $j\in\mathcal{N}$, where $v_{j^*}=\mathop{\arg}{\max}_{v\in V_k}x_{j,v}^T$. This implies that $V^*\subseteq S_i^{T+1}$ for all $i\in\mathcal{N}$, where $V^*\triangleq \{v_{1^*},\dots,v_{n^*}\}$. Starting from $S_i^{T+1}$ at time step $t=T+1$, each agent $i\in\mathcal{N}$ computes $S_i^t$ according to the following update rule:
\begin{equation}
\label{eqn:decentralized updating rule phase 2}
S_i^{t+1}=\bigcap_{j\in\mathcal{N}_i}S_j^t.
\end{equation}

We will use the following result whose proof follows directly from induction and is thus omitted for conciseness.
\begin{lemma}
\label{lemma:convergence of decentralized updating rule phase 2}
Consider the communication graph $\mathcal{G}=(\mathcal{N},\mathcal{E})$ and the update rule \eqref{eqn:decentralized updating rule phase 2} initialized with $S_i^{T+1}$ given by Eq.~\eqref{eqn:set of candidate updating elements}. Suppose Assumption~$\ref{assumption:central clock}$ holds. For each $t\ge T+1+d(\mathcal{G})$ ($t\in\mathbb{Z}$), $S_i^t=\bigcap_{j\in\mathcal{N}}S_j^{T+1}$ for all $i\in\mathcal{N}$, where $d(\mathcal{G})$ is given by Definition~$\ref{def:diameter of graph}$.
\end{lemma}

Given any $T\in\mathbb{Z}_{\ge1}$, we then set the design parameter $T'=T+1+d(\mathcal{G})$ in the sequel. After running update rule \eqref{eqn:decentralized updating rule phase 2} until time step $t=T'$ (starting from time step $t=T+1$), we have from Lemma~$\ref{lemma:convergence of decentralized updating rule phase 2}$ $S_i^{T'}=S_j^{T'}$ for all $i,j\in\mathcal{N}$. Moreover, noting from the above arguments that $V^*\subseteq S_i^{T+1}$ for all $i\in\mathcal{N}$, where $V^*=\{v_{1^*},\dots,v_{n^*}\}$, we have $V^*\subseteq S_i^{T'}$, i.e., $S_i^{T'}\neq\emptyset$. Finally, denoting $j_{\text{min}}=\mathop{\min}\{j:v_j\in S_i^{T'}\}$, each agent $i\in\mathcal{N}$ updates $\bar{S}_i^k$ as
\begin{equation}
\label{eqn:decentralized updating rule phase 3}
\bar{S}^{k+1}_i = \{v_{j_{\text{min}}}\} \cup \bar{S}^k_i.\footnote{Note that all the agents in $\mathcal{N}$ label the elements in $V$ such that $v_j$ refers to the same element in $V$ for all $j\in\{1,\dots,m\}$.}
\end{equation}
Noting that $\bar{S}_i^k=\bar{S}_j^k$ for all $i,j\in\mathcal{N}$, we obtain $\bar{S}_i^{k+1}=\bar{S}_j^{k+1}$ for all $i,j\in\mathcal{N}$. Denote $\bar{S}^{k+1}=\bar{S}_i^{k+1}$. Combining update rules \eqref{eqn:decentralized updating rule} and \eqref{eqn:set of candidate updating elements}-\eqref{eqn:decentralized updating rule phase 3} leads to the following result.
\begin{lemma} 
\label{lemma:error bound for decentralized greedy}
Consider the update rules \eqref{eqn:decentralized updating rule}, \eqref{eqn:set of candidate updating elements} and \eqref{eqn:decentralized updating rule phase 2}-\eqref{eqn:decentralized updating rule phase 3}, where \eqref{eqn:decentralized updating rule} is  initialized with $x_i^0$ given by Eq.~\eqref{eqn:initial x_t}. Suppose Assumptions~$\ref{assumption:central clock}$ and $\ref{assumption:assumption of W}$ hold and $\psi\in\mathbb{R}_{\ge0}$ satisfies condition \eqref{eqn:cond on epsilon}. Then
\begin{equation}
\label{eqn:error bound of global objective}
f(\bar{S}^{k+1})-f(\bar{S}^k)\ge\{\mathop{\max}_{v\in V_k} f(\{v\}\cup\bar{S}^k)-f(\bar{S}^k)\}-\psi-2\epsilon(T),
\end{equation}
where $\epsilon(T)=\sqrt{n}(\mu(W))^TF_h$, and $F_h=\mathop{\max}_{i\in\mathcal{N}}f_i(V)$.
\end{lemma}
\begin{proof}
Denote $v^*\triangleq\mathop{\arg}{\max}_{v\in V_k} f(\{v\}\cup\bar{S}^k)-f(\bar{S}^k)$ and  note that $\bar{S}^{k+1}\setminus\bar{S}^k=v_{j_{\text{min}}}$, where $j_{\text{min}}=\mathop{\min}\{j:v_j\in S_i^{T'}\}$. We see from Lemma~$\ref{lemma:upper bound on the estimate error}$ and the definition of $x_{j,v}^0$ that 
\begin{equation}
\label{eqn:error bound 1}
\Big|x_{i,v^*}^T-(f(\{v^*\}\cup\bar{S}^k)-f(\bar{S}^k))\Big|\le\epsilon(T),
\end{equation}
and
\begin{equation}
\label{eqn:error bound 2}
\Big|x_{i,v_{j_{\text{min}}}}^T-(f(\{v_{j_{\text{min}}}\}\cup\bar{S}^k)-f(\bar{S}^K))\Big|\le\epsilon(T).
\end{equation}
Noting that $S_i^{T'}\subseteq S_i^{T+1}$ and $v_{j_{\text{min}}}\in S_i^{T'}$, we have from update rule \eqref{eqn:set of candidate updating elements}
\begin{equation}
\label{eqn:error bound 3}
x_{i,v_{j_{\text{min}}}}^T\ge x_{i,v_{i^*}}^T-\psi,
\end{equation}
where $x_{i,v_{i^*}}^T=\mathop{\max}_{v\in V_k}x_{i,v}^T$. We then have the following:
\begin{align}\nonumber
&(f(\{v_{j_{\text{min}}}\}\cup\bar{S}^K)-f(\bar{S}^k))-(f(\{v^*\}\cup\bar{S}^k)-f(\bar{S}^k))\\
\ge& (f(\{v_{j_{\text{min}}}\}\cup\bar{S}^K)-f(\bar{S}^k))-x_{i,v^*}^T-\epsilon(T)\label{eqn:error bound derivation 1}\\
\ge& (f(\{v_{j_{\text{min}}}\}\cup\bar{S}^K)-f(\bar{S}^k))-x_{i,v_{i^*}}^T-\epsilon(T)\label{eqn:error bound derivation 2}\\
\ge& (f(\{v_{j_{\text{min}}}\}\cup\bar{S}^K)-f(\bar{S}^k))-x_{i,v_{j_{\text{min}}}}^{T}-\psi-\epsilon(T)\label{eqn:error bound derivation 3}\\
\ge& -\psi-2\epsilon(T),\label{eqn:error bound derivation 4}
\end{align}
where \eqref{eqn:error bound derivation 1} and \eqref{eqn:error bound derivation 4} follow from \eqref{eqn:error bound 1} and \eqref{eqn:error bound 2}, respectively,  \eqref{eqn:error bound derivation 3} follows from \eqref{eqn:error bound 3}, and \eqref{eqn:error bound derivation 2} follows from the fact $x_{i,v_{i^*}}^{T}=\mathop{\max}_{v\in V_k}x_{i,v}^{T}$. 
\end{proof}

In summary, after running  \eqref{eqn:decentralized updating rule}, \eqref{eqn:set of candidate updating elements} and \eqref{eqn:decentralized updating rule phase 2}-\eqref{eqn:decentralized updating rule phase 3} as described above in the $(k+1)$th round of the distributed greedy algorithm,  each agent $i\in\mathcal{N}$ obtains $x_i^{T}$, $S_i^{T'}$ and $\bar{S}_i^{k+1}$, where $T'=T+1+d(\mathcal{G})$ and $\bar{S}_i^{k+1}=\bar{S}_j^{k+1}$ for all $i,j\in\mathcal{N}$. The algorithm then enters the next round (with $k$ incremented by $1$) and repeats the same processes as described above, where all the results derived still hold.

The procedure described in this section is summarized in Algorithm~$\ref{algorithm:decentralized greedy}$, where the algorithm is implemented at each $i\in\mathcal{N}$ in a distributed way. As argued above, Algorithm~$\ref{algorithm:decentralized greedy}$ allows all the agents in $i\in\mathcal{N}$ to reach consensus at a solution $\bar{S}^K$ to Problem \eqref{eqn:global objective function}, i.e., $\bar{S}_i^K=\bar{S}^K$ $\forall i\in\mathcal{N}$.

\section{Convergence Analysis}\label{sec:submodular case}
In this section, we analyze the performance (i.e., convergence) of Algorithm~$\ref{algorithm:decentralized greedy}$. Note that the (centralized) greedy algorithm solves Problem \eqref{eqn:global objective function} in the centralized case with the multiplicative approximation ratio of $(1-1/e)$, i.e., $f(S_g)\ge(1-1/e)f(S^*)$, where $S_g$ is the solution returned by the greedy algorithm and $S^*$ is an optimal solution to Problem \eqref{eqn:global objective function}. Hence, we analyze the performance of Algorithm~$\ref{algorithm:decentralized greedy}$ by comparing the convergence of $f(\bar{S}_i^K)$ to $(1-1/e)f(S^*)$.  We will use the following result from \cite{streeter2009online}.
\begin{lemma}
\label{lemma:greedy algorithm with choice error}
Consider Problem \eqref{eqn:global objective function} in the centralized case. Denote $\bar{G}_1\triangleq\emptyset$ and $\bar{G}_j\triangleq\{\bar{g}_1,\dots,\bar{g}_{j-1}\}$ for all $j\in\{2,\dots,K+1\}$. Suppose
\begin{equation*}
f(\{\bar{g}_j\}\cup \bar{G}_j)-f(\bar{G}_j)\ge\{\mathop{\max}_{v\in V\setminus\bar{G}_j}(f(\{v\}\cup \bar{G}_j)-f(\bar{G}_j))\}-\tau_j
\end{equation*}
for all $j\in\{1,\dots,K\}$, where $\tau_j\in\mathbb{R}_{\ge0}$. Then
\begin{equation}
f(\bar{G}_{K+1})\ge (1-\frac{1}{e})f(S^*)-\sum_{k=1}^K\tau_k,
\end{equation}
where $S^*$ is an optimal solution to Problem \eqref{eqn:global objective function}.
\end{lemma}

The result below follows directly from Lemmas~$\ref{lemma:error bound for decentralized greedy}$ and $\ref{lemma:greedy algorithm with choice error}$.
\begin{theorem}
\label{thm:convergence of decentralized greedy}
Consider Algorithm~$\ref{algorithm:decentralized greedy}$ for Problem \eqref{eqn:global objective function} with a set $\mathcal{N}$ of agents. Suppose Assumptions~$\ref{assumption:central clock}$ and $\ref{assumption:assumption of W}$ hold and $\psi\in\mathbb{R}_{\ge0}$ satisfies condition \eqref{eqn:cond on epsilon}. Then Algorithm~$\ref{algorithm:decentralized greedy}$ lets all the agents in $\mathcal{N}$ reach consensus at a solution $\bar{S}^K$ to Problem  \eqref{eqn:global objective function}  that satisfies 
\begin{equation*}
f(\bar{S}^K)\ge (1-\frac{1}{e}) f(S^*)-K(\psi+2\epsilon(T)),
\end{equation*}
where $\epsilon(T)=\sqrt{n}(\mu(W))^TF_h$ with $F_h=\mathop{\max}_{i\in\mathcal{N}}f_i(V)$, and $S^*$ is an optimal solution to Problem \eqref{eqn:global objective function}.
\end{theorem}

Theorem $\ref{thm:convergence of decentralized greedy}$ shows that $\bar{S}_K$ approaches the $(1-1/e)$ approximation of $S^*$ with an additive factor $E_r\triangleq K(\psi+2\epsilon(T))$. We know from the definition of $\epsilon(T)$ that $E_r$ is a function of the number of the agents, the bound on the local functions and the design parameters (i.e., $W$, $T$ and $\psi$). In the context of Theorem $\ref{thm:convergence of decentralized greedy}$, we analyze how the additive factor (i.e., $E_r$) behaves in terms of those quantities under different scenarios. In particular, we are interested in how the additive factor depends on the number of communication steps in each round of the algorithm. First, let us consider the fast communication scenario (e.g.,  \cite{carli2008distributed}). In this scenario, agents can communicate sufficiently fast, i.e., $T\to\infty$, in each round of the distributed greedy algorithm. Since $\mu(W)<1$ from Assumption~$\ref{assumption:assumption of W}$, $\epsilon(T)\to0$ as $T\to\infty$. Moreover, the lower bound on $\psi$ in \eqref{eqn:cond on epsilon} tends to zero. Consequently, we can choose the design parameter $\psi$ to be arbitrarily close to zero and obtain $E_r\to 0$.

Next, we consider the scenario where the communication among the agents in each round of the distributed algorithm is limited. Suppose $n$ is fixed and the input to Algorithm~$\ref{algorithm:decentralized greedy}$ is also fixed, i.e., $K$ and $F_h$  are fixed. We then have $E_r=K\psi+O((\mu(W))^T)$. If we can choose the design parameter $\psi$ such that $\psi=O((\mu(W))^T)$, we obtain $E_r=O((\mu(W))^T)$, which implies that $E_r$ vanishes at an exponential rate. In contrast, if we assume that $\psi$ is fixed, we have $E_r=K\psi+O((\mu(W))^T)$, which implies that $E_r$ converges exponentially to $K\psi$.

Indeed, using techniques in, e.g., \cite{boyd2004fastest}, one can optimally choose the weight matrix $W$ such that $\mu(W)$ is minimized in the above scenarios, which leads to accelerations in the convergence rate. In summary, we observe a tradeoff between the performance of the distributed greedy algorithm and the number of communication steps in each round of the algorithm. Moreover, the performance of the algorithm also depends on the choice of $\psi$. It is also worth noting that our distributed greedy algorithm achieves exponential convergence rates (as described above), while the algorithm proposed in \cite{mokhtari2018decentralized} only achieves sublinear convergence rates.

\section{Nonsubmodular Objective Functions}
In this section, we extend our previous analysis to cases when the objective functions in Problem \eqref{eqn:global objective function} are nonsubmodular. In other words, we consider the scenario where the local objective function $f_i(\cdot)$ is monotone nondecreasing with $f_i(\emptyset)=0$, but not necessarily submodular, for all $i\in\mathcal{N}$. We first note that the (centralized) greedy algorithm has also been applied to solve Problem \eqref{eqn:global objective function} with nonsubmodular objective functions using the notion of submodularity ratio (e.g., \cite{bian2017guarantees}). 
\begin{definition}(Submodularity ratio)
\label{def:submodularity ratio}
Given a set $V$, the submodularity ratio of a nonnegative set function $f:2^{V}\to\mathbb{R}_{\ge0}$ is the largest $\gamma\in\mathbb{R}_{\ge0}$ that satisfies $\sum_{a\in A\setminus B}\big(f(\{a\}\cup B)-f(B)\big)\ge\gamma\big(f(A\cup B)-f(B)\big)$
for all $A,B\subseteq V$.
\end{definition} 
\begin{remark}
For a nonnegative and nondecreasing function set $f(\cdot)$ with submodularity ratio $\gamma$, we have $\gamma\in[0,1]$, and $f(\cdot)$ is submodular if and only if $\gamma=1$  \cite{bian2017guarantees}. 
\end{remark}

We now extend Lemma~$\ref{lemma:greedy algorithm with choice error}$ to nonsubmodular functions; a proof of the following result is included in the appendix.
\begin{lemma}
\label{lemma:greedy algorithm with choice error nonsubmodular}
Consider Problem \eqref{eqn:global objective function} in the centralized case, where the objective function $f(\cdot)$ is monotone nondecreasing with submodularity ratio $\gamma\in\mathbb{R}_{>0}$. Denote $\bar{G}_1\triangleq\emptyset$ and $\bar{G}_j\triangleq\{\bar{g}_1,\dots,\bar{g}_{j-1}\}$ for all $j\in\{2,\dots,K+1\}$. Suppose
\begin{equation*}
f(\{\bar{g}_j\}\cup \bar{G}_j)-f(\bar{G}_j)\ge\{\mathop{\max}_{v\in V\setminus\bar{G}_j}(f(\{v\}\cup \bar{G}_j)-f(\bar{G}_j))\}-\tau_j
\end{equation*}
for all $j\in\{1,\dots,K\}$, where $\tau_j\in\mathbb{R}_{\ge0}$. Then
\begin{equation}
f(\bar{G}_{K+1})\ge (1-e^{-\gamma})f(S^*)-\sum_{k=1}^K\tau_k,
\end{equation}
where $S^*$ is an optimal solution to Problem \eqref{eqn:global objective function}.
\end{lemma}

Using similar arguments to those for Theorem $\ref{thm:convergence of decentralized greedy}$, one can obtain the following result from Definition~$\ref{def:submodularity ratio}$ and Lemma~$\ref{lemma:greedy algorithm with choice error nonsubmodular}$; the proof is omitted for conciseness.
\begin{corollary}
\label{coro:convergence of decentralized greedy nonsubmodular}
Consider Algorithm~$\ref{algorithm:decentralized greedy}$ for Problem \eqref{eqn:global objective function} with a set $\mathcal{N}$ of agents. Suppose Assumptions~$\ref{assumption:central clock}$ and $\ref{assumption:assumption of W}$ hold and $\psi\in\mathbb{R}_{\ge0}$ satisfies condition \eqref{eqn:cond on epsilon}. Denote the submodularity ratio of the local objective function $f_i(\cdot)$ as $\gamma_i\in\mathbb{R}$ for all $i\in\mathcal{N}$ and denote $\gamma_c\triangleq\mathop{\min}_{i\in\mathcal{N}}\gamma_i$. Suppose $\gamma_i>0$ for all $i\in\mathcal{N}$. Then Algorithm~$\ref{algorithm:decentralized greedy}$ lets all the agents in $\mathcal{N}$ reach consensus at a solution $\bar{S}^K$ to Problem  \eqref{eqn:global objective function}  that satisfies 
\begin{equation*}
\label{eqn:convergence nonsubmodular}
f(\bar{S}^K)\ge (1-e^{-\gamma_c}) f(S^*)-K(\psi+2\epsilon(T)),
\end{equation*}
where $\epsilon(T)=\sqrt{n}(\mu(W))^TF_h$ with $F_h=\mathop{\max}_{i\in\mathcal{N}}f_i(V)$, and $S^*$ is an optimal solution to Problem \eqref{eqn:global objective function}.
\end{corollary}

Similarly to Section  \ref{sec:submodular case}, Corollary $\ref{coro:convergence of decentralized greedy nonsubmodular}$ leads to a tradeoff between the performance of the distributed greedy algorithm and the number of communication steps in each round of the algorithm, under the nonsubmodular setting.

\section{Conclusions}\label{sec:conclusion}
In this paper, we proposed a distributed greedy algorithm for maximizing a global submodular function, subject to a cardinality constraint, with a group of agents communicating over a network. The distributed greedy algorithm allows each agent to converge to within an additive factor of the $(1-1/e)$ approximation of the optimal solution to the global maximization problem. The additive factor reveals a tradeoff between the performance of the algorithm and the number of communication steps in each round of the algorithm. Finally, we extended our analysis to cases when the objective function is not submodular by leveraging the notion of submodularity ratio. 

\section*{Appendix}

\subsection*{Proof of Lemma $\ref{lemma:greedy algorithm with choice error nonsubmodular}$:}
The proof is based on the idea of the proof for Theorem $6$ in \cite{streeter2009online}.  Denote $\Delta_j\triangleq f(S^*)-f(\bar{G}_j)$ for all $j\in\{1,\dots,K+1\}$ and $\bar{\beta}_j\triangleq f(\{\bar{g}_j\cup\bar{G}_j\})-f(\bar{G}_j)=f(\bar{G}_{j+1})-f(\bar{G}_j)$ for all $j\in\{1,\dots,K\}$. We then have from Definition $\ref{def:submodularity ratio}$
\begin{equation}
\label{eqn:nonsub derivation 1}
f(\bar{G}_j\cup S^*)-f(\bar{G}_j)\le\frac{1}{\gamma}(\sum_{v\in S^*\setminus\bar{G}_j}f(\{v\}\cup\bar{G}_j)-f(\bar{G}_j)).
\end{equation}
Noting that $\bar{\beta}_j\ge\{\mathop{\max}_{v\in V\setminus\bar{G}_j}(f(\{v\}\cup \bar{G}_j)-f(\bar{G}_j))\}-\tau_j$ for all $j\in\{1,\dots,K\}$, we have $f(\{v\}\cup \bar{G}_j)-f(\bar{G}_j))\le\bar{\beta}_j+\tau_j$ for all $v\in S^*\setminus\bar{G}_j$ and for all $j\in\{1,\dots,K\}$. It then follows from $|S^*|=K$ and \eqref{eqn:nonsub derivation 1} that
\begin{align}\nonumber
& f(S^*)\le f(\bar{G}_j\cup S^*)\le f(\bar{G}_j)+\frac{K}{\gamma}(\bar{\beta}_j+\tau_j)\\\nonumber
\Rightarrow&\Delta_j\le\frac{K}{\gamma}(\bar{\beta}_j+\tau_j)\Rightarrow\Delta_j\le\frac{K}{\gamma}(\Delta_j-\Delta_{j+1}+\tau_j)\\
\Rightarrow&\Delta_{j+1}\le(1-\frac{\gamma}{K})\Delta_j+\tau_j.\label{eqn:nonsub derivation 2}
\end{align}
Unrolling \eqref{eqn:nonsub derivation 2}, we obtain $\Delta_{K+1}\le(1-\frac{\gamma}{K})^K\Delta_1+\sum_{j=1}^K\tau_j$, where we use the fact $1-\frac{\gamma}{K}<1$. Therefore, $f(S^*)-f(\bar{G}_{K+1})\le(1-\frac{\gamma}{K})^Kf(S^*)+\sum_{j=1}^K\tau_j$, which implies $f(\bar{G}_{K+1})\ge f(S^*)-e^{-\gamma}f(S^*)-\sum_{j=1}^K\tau_j$.\hfill\QED

\bibliographystyle{IEEEtran}
\bibliography{main}

\end{document}